\documentclass{article}
\usepackage[utf8]{inputenc}
\usepackage{color}
\usepackage{amsfonts,amsmath,amsthm,amssymb}
\usepackage{mathrsfs}
\usepackage{geometry}
\theoremstyle{plain}
\newtheorem{theorem}{Theorem}[section]
\newtheorem{corollary}[theorem]{Corollary}
\newtheorem{lemma}[theorem]{Lemma}

\usepackage{hyperref}
\newtheorem{definition}[theorem]{Definition}
\newtheorem{example}[theorem]{Example}
\newtheorem{remark}[theorem]{Remark}
\newtheorem{assumptions}[theorem]{Assumptions}

\newcommand{\dx}{\textnormal{d}\mathbf{X}}
\newcommand{\Id}{\mbox{Id}}
\newcommand{\gubnorm}[2]{\left\|#1,#2 \right\|_{X,2\alpha,\gamma}}
\newcommand{\gubnormpar}[3]{\left\|#1,#2 \right\|_{X,2\alpha,#3}}
\newcommand{\xx}{X^{(2)}}
\newcommand{\bX}{\mathbf{X}}
\newcommand{\xxt}{\tilde{X}^{(2)}}
\newcommand{\IN}{\mathbb{N}}

\title{Global solutions for semilinear rough partial differential equations}
\author{Robert Hesse\thanks{Friedrich Schiller University Jena, Department of Mathematics, Ernst-Abbe-Platz 2, 07743 Jena, Germany. E-Mail: robert.hesse@uni-jena.de.} 
~~
and Alexandra Neam\c tu\thanks{University of Konstanz, Department of Mathematics and Statistics,  Universit\"atsstr.~10, 78464 Konstanz, Germany. E-Mail: alexandra.neamtu@uni-konstanz.de}}

\def\txtd{{\textnormal{d}}}

\def\cO{\mathcal{O}}

\def\cB{\mathcal{B}}

\def\cL{\mathcal{L}}
\def\cX{\mathcal{X}}
\begin{document}

\maketitle

\begin{abstract}
    We construct global-in-time solutions for semilinear parabolic rough  partial differential equations. We work on a scale of Banach spaces tailored to the controlled rough path approach and derive suitable a-priori estimates of the solution which do not contain quadratic terms.
\end{abstract}



\section{Introduction}

The main goal of this work is to advance the theory of global solutions for semilinear parabolic rough partial differential equations (RPDEs). Since the breakthrough in rough paths theory for stochastic ordinary differential equations,
there has been a strong interest in investigating rough path approaches for partial differential equations. However, there are few results regarding global well-posedness of solutions for partial differential equations perturbed by nonlinear rough multiplicative noise. We contribute to this aspect and establish global-in-time solutions for semilinear parabolic RPDEs. We fix a time horizon $T>0$ and consider on a separable Banach space $(\cB,|\cdot|)$ the rough evolution equation
\begin{align}\label{equation:main}
\begin{cases}
\txtd y_t = [A y_t + F (y_t)]~\txtd t + G(y_t)~\dx_t,~~t\in[0,T]\\
y(0)=y_0\in \cB.
\end{cases}
\end{align}
We assume that the linear operator $A$ generates an analytic $C_0$-semigroup $(S(t))_{t\geq 0}$ on $\cB$ and the noise $\bX=(X,\xx)$ is a finite-dimensional $\alpha$-H\"older rough path~\cite{FrizHairer}, for $\alpha\in(\frac{1}{3},\frac{1}{2})$, as specified below. A famous example is constituted by fractional Brownian motion with Hurst index $H\in(1/3,1/2]$.  The drift term $F$ and the diffusion coefficient $G$ satisfy suitable smoothness conditions. \\
Several approaches have been used in order to investigate RPDEs. For instance, for RPDEs perturbed by transport noise, solutions satisfying energy estimates have been constructed using the notion of unbounded rough drivers and the rough Gronwall lemma~\cite{HH,HNS20, Hofmanova1}.
On the other hand for RPDEs perturbed by nonlinear multiplicative noise, the semigroup approach has been employed by~\cite{DeyaGubinelliTindel, GubinelliTindel, GHairer, GHN, HN19} and the references specified therein. In this setting, the main task is to define the rough convolution
$\int_0^t S(t-s)G(y_s)~\dx_s$. To this aim one first of all needs the notion of a controlled rough path~\cite{Gubinelli}, which is a pair $(y,y')$ of $\alpha$-H\"older continuous functions
satisfying an abstract Taylor-like expansion in terms of H\"older regularity
given by
\begin{align*}
y_t= y_s+y'_s X_{s,t}+ R^{y}_{s,t},
\end{align*}
where the remainder $R^y_{s,t}$ is supposed to be $2\alpha$-H\"older-regular. Due to the lack of regularity of the semigroup $(S(t))_{t\geq 0}$ in zero, it is a challenging task to find an appropriate meaning of a controlled rough path. The main idea is to consider controlled rough paths on a scale of Banach spaces $(\cB_\gamma)_{\gamma\in\mathbb{R}}$ satisfying the interpolation inequality 
\begin{align*}
    |x|^{\gamma-\theta}_\beta \lesssim |x|^{\gamma-\beta}_\theta |x|^{\beta-\theta}_\gamma,
\end{align*}
which holds for $\theta\leq \beta\leq \gamma$ and $x\in \cB_\gamma$~\cite{Lunardi}.
The advantage of this approach is that it allows one to view the semigroup as a bounded operator on all these spaces and exploit space-time regularity specific to the parabolic setting. Such an approach was exploited in~\cite{GHN} in the context of non-autonomous RPDEs and in~\cite{GHairer},
where the semigroup was directly incorporated in the definition of the controlled rough path. \\
However, global well-posedness results for RPDEs are more complicated to obtain, due to the quadratic terms which occur in the a-priori estimates of the solution. These arise in the composition of a controlled rough path $(y,y')$ with a smooth function $G$ which is naturally given by $(G(y),DG(y)y')$, see~\cite{FrizHairer}.
This operation involves Taylor expansions of the nonlinear term $G$ yielding a quadratic estimate for the norm of $(G(y),DG(y)y')$ in terms of the controlled rough path norm of $(y,y')$. Therefore a-priori estimates by a direct application of the Gronwall lemma are not possible. This issue was solved only under certain boundedness assumptions on the diffusion coefficient, see for example~\cite{HN20,Hofmanova1}. In this work we additionally incorporate a drift term in~\eqref{equation:main} that satisfies a linear growth condition and impose a boundedness restriction on the diffusion coefficient to derive global-in-time existence of solutions. \\
The global well-posedness of RPDEs is a crucial step in studying their long-time behavior. For instance, there are results regarding the existence of random dynamical systems generated by RPDEs with transport~\cite{Hofmanova1,Hofmanova2}, nonlinear multiplicative~\cite{HN20} and nonlinear conservative noise~\cite{Gess}. Since the solutions are constructed in a pathwise sense, the usual issue with nullsets from the theory of random dynamical systems~\cite{Arnold} does not occur in this approach. Therefore, the existence of a random dynamical system (Theorem~\ref{thm:rds}) is an immediate consequence of our main result (Theorem~\ref{global:sol}).\\
This work is structured as follows. In Section~\ref{sec:prelim} we collect important properties regarding rough paths and analytic $C_0$-semigroups on interpolation spaces. Section~\ref{sec:main} contains our main results regarding the existence of a global-in-time solutions for semilinear rough partial differential equations. To this aim we provide a suitable estimate of the controlled rough integral together with an a-priori bound of the solution, which does not contain quadratic terms. This is obtained using the structure of the solution of an RPDE and imposing certain boundedness restrictions on the diffusion coefficient. It would be desirable to extend the global-in-time existence of solutions to RPDEs with a dissipative drift, as considered in the finite-dimensional case in~\cite{Weber}. We present some applications in Section~\ref{sec:ex}.

\section{Preliminaries}\label{sec:prelim}
We first provide some fundamental concepts from rough path theory starting with the definition of a $d$-dimensional $\alpha$-H\"older rough path.

\begin{definition}\label{hrp}\emph{($\alpha$-H\"older rough path)}
	Let $J \subset \mathbb{R}$ be  a compact interval. We call a pair $\mathbf{X}=(X,\xx)$ $\alpha$-H\"older rough path if $X\in C^{\alpha}(J, \mathbb{R}^d)$ and $\xx\in C^{2\alpha}(\Delta_{J}, \mathbb{R}^d\otimes \mathbb{R}^d)$, where $\Delta_{J} := \left\{\left(s,t \right) \in J^2 \colon  s \leq t \right\}$.   
 Furthermore $X$ and $\xx$ are connected via Chen's relation, meaning that
	\begin{align}\label{chen}
	\xx_{s,t} - \xx_{s,u} - \xx_{u,t} = (X_{u} - X_{s})\otimes ( X _{t}- X_{u}) 
	,~~ \mbox{for } s,u,t \in J,~~ s \leq u \leq t .
	\end{align}
	In the literature $\xx$ is referred to as L\'evy-area or second order process.	
\end{definition}
Throughout this manuscript, we assume for simplicity that $d=1$ and further introduce an appropriate distance between two $\alpha$-H\"older rough paths.
\begin{definition}
	Let $J\subset\mathbb{R}$ be a compact interval and let $\mathbf{X}$ and $\mathbf{\tilde{X}}$ be two $\alpha$-H\"older rough paths. We introduce the $\alpha$-H\"older rough path (inhomogeneous) metric
	\begin{align}\label{rp:metric}
	d_{\alpha,J}(\mathbf{X},\mathbf{\tilde{X}} )
	:= \sup\limits_{(s,t)\in \Delta_J} \frac{|X_{t}-X_{s}-\tilde{X}_{t}+\tilde{X}_{s}|}{|t-s|^{\alpha}} 
	+ \sup\limits_{(s,t) \in \Delta_{J}}
	\frac{|\xx_{s,t}-\xxt_{s,t}|} {|t-s|^{2\alpha}}.
	\end{align}
	We set $\rho_\alpha(\mathbf{X}):=d_{\alpha,[0,T]}(\mathbf{X},0)$.
\end{definition}
For more details on this topic consult~\cite[Chapter 2]{FrizHairer}. We stress that in our situation we always have that $X_0=0$ and therefore~\eqref{rp:metric} is a metric.\\
Throughout this manuscript $C$ stands for a universal constant which varies from line to line. We write $a \lesssim b$ if there exists a constant $C>0$ such that $a\leq C b$. The constant $C$ can depend on the parameters $\alpha,\gamma,\rho_\alpha(\bX)$ as well as on $F$ and $G$ and their derivatives but it is independent of the initial data $y_0$. Moreover it can also depend on time but it is uniformly with respect to $T$ on compact intervals.\\

Since we consider parabolic RPDEs, we work with the following function spaces similar to~\cite{GHairer,GHN}.
\begin{definition}\label{raum}
A family of separable Banach spaces $(\cB_\theta,|\cdot|_\theta)_{\theta\in\mathbb{R}}$ is called a monontone family of interpolation spaces if for $\beta_1\leq \beta_2$, the space $\cB_{\beta_2}\subset \cB_{\beta_1}$ with dense and continuous embedding and the following interpolation inequality holds for $\theta\leq \beta\leq \gamma$ and $x\in  \cB_{\gamma}$:
\begin{align}\label{interpolation:ineq}
    |x|^{\gamma-\theta}_\beta \lesssim |x|^{\gamma-\beta}_\theta |x|^{\beta-\theta}_\gamma.
\end{align}
\end{definition}
The main advantage of this approach is that we can view the semigroup $(S(t))_{t\geq 0}$ as a linear mapping between these interpolation spaces and obtain the following standard bounds for the corresponding operator norms. If $S:[0,T]\to \cL(\cB_{\gamma},\cB_{ \gamma+1})$ is such that for every $x\in \cB_{\gamma+1}$ and $t\in(0,T]$ we have that $|(S(t)-\Id)x|_{\gamma} \lesssim t|x|_{\gamma+1}$ and $|S(t)x|_{\gamma+1}\lesssim t^{-1} |x|_{\gamma}$, then for every $\sigma\in[0,1]$ we have that $S(t)\in\cL(\cB_{\gamma+\sigma})$ and 
\begin{align}
|(S(t)-\Id) x|_{\gamma}&\lesssim t^\sigma |x|_{\gamma+\sigma}\label{hg:1}\\
 |S(t)x|_{\gamma+\sigma}&\lesssim t^{-\sigma}|x|_\gamma\label{hg:2}.
\end{align}
For further details regarding these interpolation spaces, see~\cite{Lunardi}. We emphasize that $\alpha\in(\frac{1}{3},\frac{1}{2})$ always indicates the time-regularity of the random input, while $\gamma$ stands for the spatial regularity in $\cB_\gamma$. We work with mild solutions for~\eqref{equation:main} which are given by the variation of constants formula
\begin{align}\label{mild}
    y_t = S(t)y_0 + \int_0^t S(t-s)F(y_s)~\txtd s + \int_0^t S(t-s) G(y_s)~\dx_s.
\end{align}
In order to construct the rough integral $\int_0^t S(t-s)G(y_s)~\dx_s$ and give a proper meaning of the mild formulation~\eqref{mild}, we introduce the following space of controlled rough paths. This incorporates suitable space-time regularity of the solution reflecting the parabolic nature of the problem we consider, similar to~\cite{GHN}.
\begin{definition}\label{def:crp}
We call a pair $(y,y')$ a controlled rough path if
\begin{itemize}
    \item $(y,y')\in C([0,T];\cB_\gamma) \times (C[0,T];\cB_{\gamma-\alpha} ) \cap C^{\alpha}([0,T];\cB_{\gamma-2\alpha} )$. The component $y'$ is referred to as the Gubinelli derivative\footnote{For smooth paths $y$ and $X$, the choice of $y'$ is not unique. However, one can show that for rough inputs $X$, $y'$ is uniquely determined by $y$, see~\cite[Remark~4.7 and Section~6.2]{FrizHairer}.} of $y$.
    \item the remainder  \begin{align}\label{remainder}
    R^y_{s,t}= y_{s,t} -y'_s X_{s,t}    
    \end{align}
     belongs to $ C^{\alpha}([0,T];\cB_{\gamma-\alpha})\cap C^{2\alpha}([0,T];\cB_{\gamma-2\alpha})$.
\end{itemize}
\end{definition}

The space of controlled rough paths is denoted by $D^{2\alpha}_{X,\gamma}$ and endowed with the norm
\begin{align}\label{g:norm}
    \gubnorm{y}{y'}= \left\|y \right\|_{\infty,\cB_\gamma} 
    + \|y' \|_{\infty,\cB_{\gamma-\alpha}}
    + \left\|y'\right\|_{\alpha,\cB_{\gamma-2\alpha}}
    + \left\|R^y \right\|_{\alpha,\cB_{\gamma-\alpha}}+ \left\|R^y \right\|_{2\alpha,\cB_{\gamma-2\alpha}}.
\end{align}
\begin{remark}
\begin{itemize}
    \item [1).] Note that we do not make the H\"older continuity of $y$ as part of the definition of a controlled rough path, since
using~\eqref{remainder} one immediately obtains for $\theta \in \left\{\alpha,2\alpha \right\}$ that
\begin{align}\label{est:hoelder:y}
    \left\|y \right\|_{\alpha,\cB_{\gamma-\theta}}
    \leq \left\|y' \right\|_{\infty,\cB_{\gamma-\theta}} \left\|X\right\|_{\alpha} + \left\|R^y \right\|_{\alpha,\cB_{\gamma-\theta}}.
\end{align}
\item [2).] In order to emphasize the time horizon that we consider we write $D^{2\alpha}_{X,\gamma}([0,T])$ instead of $D^{2\alpha}_{X,\gamma}$.
\end{itemize}
\end{remark}
\begin{remark}
Definition~\ref{def:crp} states that $(y,y')\in D^{2\alpha}_{X,\gamma}$ is controlled by $X$ according to the monotone family of interpolation spaces $(\cB_\gamma)_{\gamma\in\mathbb{R}}$ as in~\cite{GHN}. One can make the semigroup $(S(t))_{t\geq 0}$ part of the definition of the controlled rough path as in~\cite{GHairer}. We work with Definition~\ref{def:crp}, since it incorporates the space-time regularity of the solution and stays closer to the finite-dimensional setting~\cite{FrizHairer, Gubinelli}.
\end{remark}
We state the main assumptions on the coefficients of~\eqref{equation:main} which ensure the global-in-time existence of solutions.
\begin{assumptions}
\end{assumptions}
\begin{itemize}
\item [$(y_0)$] The initial condition $y_0\in\cB_\gamma$.
    \item [(F)]  The nonlinear drift term $F:\cB_\gamma\to \cB_{\gamma-\delta}$ for $\delta\in[0,1)$ is locally Lipschitz continuous with linear growth condition. 
    \item [(G)] Let $\theta\in\{0,\alpha,2\alpha\}$ and $0\leq\sigma<\alpha$. The nonlinear diffusion coefficient $G:\cB_{\gamma-\theta}\to\cB_{\gamma-\theta-\sigma}$ is three times continuously differentiable with bounded derivatives, i.e. $\|D^k G\|_{\cL(\cB^{\otimes k}_{\gamma-\theta},\cB_{\gamma-\theta-\sigma}) } <\infty$ for $k\in\{1,2,3\}$
    and 
    the derivative of
    \begin{align}\label{ass:g}
        DG(\cdot) G(\cdot) \colon \cB_{\gamma-\alpha} \to \cB_{\gamma-2\alpha-\sigma}
    \end{align}
    is bounded.
    \begin{remark}
        \begin{itemize}
        \item [1).] Note that this condition is valid if $G$ itself is bounded or linear.
        \item [2).] Moreover, assumption (G) implies the following Lipschitz property 
        \begin{align}\label{est:g}
            |\big(DG(y^1) - DG(y^2)\big)G(y^1)|_{\cB_{\gamma-2\alpha-\sigma}}\lesssim |y^1-y^2|_{\cB_{\gamma-\alpha}}, \mbox{ for } y^1,y^2\in \cB_{\gamma-\alpha},
        \end{align}
        since
        \begin{align*}
         |\big(DG(y^1) - DG(y^2)\big)G(y^1)|_{\cB_{\gamma-2\alpha-\sigma}}   &\leq  |DG(y^1) G(y^1) - DG(y^2) G(y^2)|_{\cB_{\gamma-2\alpha-\sigma}}\\ 
         &+  | DG(y^2)\big(G(y^1)-G(y^2)\big)|_{\cB_{\gamma-2\alpha-\sigma}}.
        \end{align*}
    \end{itemize}
    \end{remark}
\end{itemize}
\section{Main result}\label{sec:main}

According to~\cite[Theorem 5.1]{GHN} we know that the SPDE~\eqref{equation:main} has a local-in-time solution.
For the sake of completeness we provide two results established in~\cite{GHN} regarding the construction of the rough integral and the existence of the local solution.
The following lemma (\cite[Theorem 4.5]{GHN}) contains the construction of the rough integral.
\begin{lemma}
Let $\bX$ be an $\alpha$-H\"older rough path and let $(y,y') \in D^{2\alpha}_{X,\gamma}$. Then the rough integral 
\begin{align*}
   \int_0^t S(t-r)y_r~\dx_r 
   := \lim_{|\pi| \to 0 } \sum\limits_{[u,v]\in \pi}
   S(t-u) \big[y_u X_{u,v} + y'_u \xx_{u,v} \big],
\end{align*}
exists in $\cB_{\gamma-2\alpha}$, where the limit over partitions $\pi$ of $[0,t]$ is independent of the concrete choice of these partitions.
Furthermore, for all $0\leq \beta< 3\alpha$ the following bound holds true
\begin{align}\label{est:sewing:cor}
\Bigg|\int_s^t S(t-r)y_r~\dx_r - S(t-s)y_s X_{s,t} - S(t-s)y'_s\xx_{s,t}\Bigg|_{\cB_{\gamma-2\alpha+\beta} }
\lesssim \gubnorm{y}{y'} (t-s)^{3\alpha-\beta}.
\end{align}
for all $0\leq s < t \leq T$.
\end{lemma}
The following theorem (\cite[Theorem 5.1]{GHN}) ensures the existence of a local-in-time solution.
\begin{theorem}\label{thm:local}
Let $T>0$, $F$ and $G$ satisfy the assumptions (F) and (G), $\bX=(X,\xx)$ be an $\alpha$-H\"older rough path  and $y_0 \in \cB_{\gamma}$ with $|y_0|_{\cB_{\gamma}}\leq \rho$. Then there exists  $T^*=T^*(\alpha,\gamma,\rho,X,F,G)\in (0,T]$ such that there exists a unique solution $(y,y') \in D^{2\alpha}_{X,\gamma}([0,T^*])$ up to time $T^*$ satisfying
\begin{align}\label{solution}
     (y_t,y'_t) = \Bigg(S(t)y_0 + \int_0^t S(t-s)F(y_s)~\txtd s + \int_0^t S(t-s)G(y_s)~\dx_s, G(y_t)\Bigg) \in D^{2\alpha}_{X,\gamma},
\end{align}
\end{theorem}

Using a-priori estimates we show that this solution is global-in-time provided that $F$ and $G$ satisfy the assumptions (F) and (G). We now derive the necessary a-priori estimates starting with the initial data. 
\begin{lemma}\label{lemma:intitial}
Let $y_0 \in \cB_{\gamma}$. Then $(S(\cdot)y_0,0) \in D^{2\alpha}_{X,\gamma}$ and
\begin{align*}
    \gubnorm{S(\cdot)y_0}{0}
    \lesssim \left|y_0 \right|_{\cB_{\gamma}}.
\end{align*}
\end{lemma}
\begin{proof}
By \eqref{g:norm} we have
\begin{align*}
    \gubnorm{S(\cdot)y_0}{0} 
    = \left\|S(\cdot)y_0 \right\|_{\infty,\cB_\gamma} 
    + \left\|S(\cdot)y_0 \right\|_{\alpha,\cB_{\gamma-\alpha}}
    + \left\|S(\cdot)y_0 \right\|_{2\alpha,\cB_{\gamma-2\alpha}}.
\end{align*}
Clearly,
\begin{align*}
    \left\|S(\cdot)y_0 \right\|_{\infty,\gamma}
    \lesssim \left|y_0\right|_{\cB_{\gamma}}.
\end{align*}
Further, for $\theta \in \left\{\alpha,2\alpha \right\}$ we obtain
\begin{align*}
    \left\|S(\cdot)y_0 \right\|_{\theta,\gamma-\theta} 
    \lesssim \left\|S(\cdot)\right\|_{\theta,\cL(\cB_{\gamma},\cB_{\gamma-\theta})} \left| y_0\right|_{\cB_{\gamma}}
    \lesssim \left|y_0 \right|_{\cB_{\gamma}}.
\end{align*}
\end{proof}
\begin{lemma}\label{drift}
Let $(y,y')\in D^{2\alpha}_{X,\gamma}$. Then $\Big( \int_0^t S(t-s)F(y_s)~\txtd s, 0 \Big)_{t \in [0,T]}\in D^{2\alpha}_{X,\gamma}$ and satisfies the following bound
\begin{align}\label{est:drift}
    \gubnorm{\int_0^\cdot S(\cdot-s)F(y_s)~\txtd s}{0}
    \lesssim (1+\|y\|_{\infty,\gamma}) T^{1-\delta}.
\end{align}
\end{lemma}
\begin{proof} Since the Gubinelli derivative of the deterministic integral is zero, we compute
\begin{align*}
    \gubnorm{\int_0^\cdot S(\cdot-s)F(y_s)~\txtd s}{0}
   & = \left\|\int_0^\cdot S(\cdot-s)F(y_s)~\txtd s\right\|_{\infty,\cB_\gamma} 
    + \left\| \int_0^\cdot S(\cdot-s)F(y_s)~\txtd s \right\|_{\alpha,\cB_{\gamma-\alpha}}\\
    &+ \left\| \int_0^\cdot S(\cdot-s)F(y_s)~\txtd s \right\|_{2\alpha,\cB_{\gamma-2\alpha}}  
\end{align*}
We begin with the first term and get due to the fact that $F:\cB_\gamma\to\cB_{\gamma-\delta}$ the estimate
\begin{align*}
    \left|\int_0^t S(t-s)F(y_s)~\txtd s \right|_{\cB_\gamma}
    \lesssim \int_0^t (t-s)^{-\delta} \left|F(y_s) \right|_{\cB_{\gamma-\delta}}~\txtd s \lesssim T^{1-\delta} (1+\|y\|_{\infty,\gamma})
\end{align*}
For the H\"older norms we use
\begin{align*}
\int_0^t S(t-r)F(y_r)~\txtd r- \int_0^s S(s-r)F(y_r)~\txtd r = (S(t-s)-\Id) \int_0^s S(s-r)F(y_r)~\txtd r + \int_s^t S(t-r)F(y_r)~\txtd r
\end{align*}
to obtain for all $\theta\in\{\alpha,2\alpha\}$
\begin{align*}
    \left|\int_s^t S(t-r)F(y_r)~\txtd r \right|_{\cB_{\gamma-\theta}} 
    \lesssim \int_s^t (t-r)^{(\theta-\delta)\wedge 0}\left|F(y_r)\right|_{\cB_{\gamma-\delta}}~\txtd r
    \lesssim (t-s)^{1+(\theta-\delta)\wedge 0} (1+\|y\|_{\infty,\gamma})
\end{align*}
as well as 
\begin{align*}
    \left|(S(t-s)-\Id) \int_0^s S(s-r)F(y_r)~\txtd r \right|_{\cB_{\gamma-\theta}}
    \lesssim (t-s)^{\theta} \left|\int_0^s S(s-r)F(y_r)~\txtd r \right|_{\cB_{\gamma}}
    \lesssim (t-s)^{\theta} T^{1-\delta} (1+\|y\|_{\infty,\gamma}).
\end{align*}
Thus, 
\begin{align*}
    \left\| \int_0^\cdot S(\cdot-s)F(y_s)~\txtd s \right\|_{\theta,\cB_{\gamma-\theta}} 
    \lesssim T^{1-\delta} (1+\|y\|_{\infty,\gamma}).
\end{align*}
Putting these estimates together proves~\eqref{est:drift}.
\end{proof}
We focus on the rough integral, see~\cite{GubinelliTindel, GHairer,GHN} for similar results. Here we show that the rough convolution increases the spatial regularity by $\sigma$ providing an estimate for the norm~\eqref{g:norm} of the controlled rough integral using the interpolation inequality~\eqref{est:hoelder:y}.
\begin{lemma}\label{rough:integral}
Let $(y,y')\in D^{2\alpha}_{X,\gamma}$. Then for all $0\leq \sigma < \alpha$
\begin{align}
(z,z')=\Big( \int_0^\cdot S(\cdot-s)y_s~\dx_s, y \Big) \in D^{2\alpha}_{X,\gamma+\sigma}
\end{align}
and the following estimate holds true
\begin{align}\label{est:roughintegral:y}
\gubnormpar{z}{z'}{\gamma+\sigma} \lesssim |y_0|_{\cB_\gamma} + |y'_0|_{\cB_{\gamma-\alpha}} + T^{\alpha-\sigma} \gubnorm{y}{y'}.
\end{align}
\end{lemma}
\begin{proof}
By the definition of the norm~\eqref{g:norm} and regarding that $z'=y$ we have 
\begin{align}\label{norm:z}
    \gubnormpar{z}{z'}{\gamma+\sigma}
    = \left\|z \right\|_{\infty,\cB_{\gamma+\sigma}} 
    + \left\|y \right\|_{\infty,\cB_{\gamma+\sigma-\alpha}}
    + \left\|y \right\|_{\alpha,\cB_{\gamma+\sigma-2\alpha}}
    + \left\|R^z \right\|_{\alpha,\cB_{\gamma+\sigma-\alpha}}
    + \left\|R^z \right\|_{2\alpha,\cB_{\gamma+\sigma-2\alpha}}.
\end{align}
By \eqref{est:hoelder:y} we know that $y \in C^\alpha([0,T];\cB_{\gamma-\alpha})$.
Using the interpolation inequality~\eqref{interpolation:ineq} for the scale of Banach spaces $\cB_\gamma$ we derive 
\begin{align*}
    \left| y_t-y_s \right|_{\cB_{\gamma+\sigma-\alpha}}
    \lesssim \left| y_t-y_s \right|_{\cB_{\gamma}}^{\frac{\sigma}{\alpha}} \left| y_t-y_s \right|_{\cB_{\gamma-\alpha}}^{\frac{\alpha-\sigma}{\alpha}}.
\end{align*}
Consequently, this leads to
\begin{align*}
\left\|y \right\|_{\alpha-\sigma,\gamma+\sigma-\alpha}
\lesssim \left\|y \right\|_{\infty,\gamma}^{\frac{\sigma}{\alpha}} 
 \left\|y \right\|_{\alpha,\gamma-\alpha}^{\frac{\alpha-\sigma}{\alpha}}
\end{align*}
Hence, for the second term in \eqref{norm:z} we obtain
\begin{align}\label{est:y:infty}
    \left\|y \right\|_{\infty,\gamma+\sigma-\alpha}
    \leq |y_0|_{\cB_{\gamma+\sigma-\alpha}} 
    + T^{\alpha-\sigma}\left\|y \right\|_{\alpha-\sigma,\gamma+\sigma-\alpha}
    \lesssim |y_0|_{\cB_{\gamma}} 
    + T^{\alpha-\sigma} \gubnorm{y}{y'}.
\end{align}
Similarly, for the third term of \eqref{norm:z} we apply \eqref{est:hoelder:y}
\begin{align*}
 \left\|y \right\|_{\alpha,\gamma+\sigma-2\alpha}
    \leq \left\|y' \right\|_{\infty,\gamma+\sigma-2\alpha} \left\|X\right\|_{\alpha} + \left\|R^y \right\|_{\alpha,\gamma+\sigma-2\alpha},
\end{align*}
where
\begin{align}
    \left\|y' \right\|_{\infty,\gamma+\sigma-2\alpha}
    &\leq |y'_0|_{\cB_{\gamma+\sigma-2\alpha}} + T^{\alpha-\sigma} \left\|y' \right\|_{\alpha-\sigma,\gamma+\sigma-2\alpha} \notag\\
    &\lesssim |y'_0|_{\cB_{\gamma-\alpha}} 
    + T^{\alpha-\sigma}\left\|y' \right\|_{\infty,\gamma-\alpha}^{\frac{\sigma}{\alpha}}
    \left\|y' \right\|_{\alpha,\gamma-2\alpha}^{\frac{\alpha-\sigma}{\alpha}}, \label{est:hoelder:yprime}
\end{align}
and 
\begin{align*}
\left\|R^y \right\|_{\alpha,\gamma+\sigma-2\alpha}
\leq \left\|R^y \right\|_{\alpha,\gamma-\alpha}^{\frac{\sigma}{\alpha}}
\left\|R^y \right\|_{2\alpha,\gamma-2\alpha}^{\frac{\alpha-\sigma}{\alpha}} T^{\alpha-\sigma}.
\end{align*}
Thus,
\begin{align*}
     \left\|y \right\|_{\alpha,\gamma+\sigma-2\alpha} 
     \lesssim |y'_0|_{\cB_{\gamma-\alpha}} + T^\alpha \gubnorm{y}{y'}. 
\end{align*}
For the first term of \eqref{norm:z} we use
\begin{align*}
    z_t &= \int_0^t S(t-r)y_r~\dx_r \\&= \int_0^t S(t-r)y_r~\dx_r - S(t) y_0 X_{0,t} - S(t) y'_0 \xx_{0,t} \\
    &+ S(t) y_0 X_{0,t} + S(t) y'_0 \xx_{0,t}.
\end{align*}
For the first term we apply \eqref{est:sewing:cor}
\begin{align*}
    \left|\int_0^t S(t-r)y_r~\dx_r - S(t) y_0 X_{0,t} - S(t) y'_0 \xx_{0,t} \right|_{\cB_{\gamma+\sigma}}
    \lesssim \gubnorm{y}{y'} t^{\alpha-\sigma}.
\end{align*}
For the second term one sees
\begin{align*}
    |S(t) y_0 X_{0,t}|_{\cB_{\gamma+\sigma}} 
    \lesssim t^\alpha |S(t)|_{\cL(\cB_{\gamma},\cB_{\gamma+\sigma})} 
    |y_0|_{\cB_{\gamma}}
    \lesssim t^{\alpha-\sigma} |y_0|_{\cB_{\gamma}} 
\end{align*}
Analogously 
\begin{align*}
    |S(t) y'_0\xx_{0,t} |_{\cB_{\gamma+\sigma}} \lesssim t^{2\alpha} |S(t)|_{\cL(\cB_{\gamma-\alpha},\cB_{\gamma+\sigma})}  \|\xx\|_{2\alpha} |y'_0|_{\cB_{\gamma-\alpha}} \lesssim t^{\alpha-\sigma} |y'_0|_{\cB_{\gamma-\alpha}}.
\end{align*}
In conclusion we can bound $\|z\|_{\infty,\cB_{\gamma+\sigma}}$ as 
\begin{align}\label{est:infty:z}
\|z\|_{\infty,\cB_{\gamma+\sigma}} \lesssim T^{\alpha-\sigma} \gubnorm{y}{y'}. 
\end{align}
For the remainder terms we use 
\begin{align*}
    R^z_{s,t}&= \int_0^t S(t-r)y_r~\dx_r - \int_0^s S(s-r)y_r~\dx_r - y_s X_{s,t} \\
    &= \int_s^t S(t-r)y_r~\dx_r - S(t-s)y_s X_{s,t} - S(t-s)y'_s \xx_{s,t} \\
    &+ (S(t-s)-\Id) y_s X_{s,t}
    + (S(t-s)-\Id) \int_0^{s} S(s-r)y_r~\dx_r
    +S(t-s) y'_s \xx_{s,t}.
\end{align*}
Throughout the following computations we set $\theta\in\{\alpha,2\alpha\}$.
The first term can be estimated using~\eqref{est:sewing:cor} 
\begin{align*}
    \left|\int_s^t S(t-r)y_r~\dx_r - S(t-s)y_s X_{s,t} - S(t-s)y'_s \xx_{s,t} \right|_{\cB_{\gamma+\sigma-\theta}}
    \lesssim \gubnorm{y}{y'} (t-s)^{\alpha-\sigma+\theta}.
\end{align*}
Furthermore we obtain for the second term 
\begin{align*}
    \left|(S(t-s)-\Id) y_s X_{s,t} \right|_{\cB_{\gamma+\sigma-\theta}} &\lesssim (t-s)^\alpha \|X\|_{\alpha}  | S(t-s)-\Id|_{\cL(\cB_{\gamma+\sigma-\alpha},\cB_{\gamma+\sigma-\theta})} |y_s|_{\cB_{\gamma+\sigma-\alpha}}  \\
    &\lesssim (t-s)^\theta \left|y_s \right|_{\cB_{\gamma+\sigma-\alpha}}\\
    & \lesssim (t-s)^\theta \|y\|_{\infty,\cB_{\gamma+\sigma-\alpha}}
\end{align*}
which was estimated in~\eqref{est:y:infty}.\\
For the second term we get
\begin{align*}
    \left|(S(t-s)-\Id)\int_0^{s} S(s-r)y_r~\dx_r\right|_{\cB_{\gamma+\sigma-\theta}}&\lesssim |S(t-s)-\Id|_{\cL(\cB_{\gamma+\sigma},\cB_{\gamma+\sigma-\theta})} \left| \int_0^s S(s-r)y_r~\dx_r \right|_{\cB_{\gamma+\sigma}} \\ 
    &\lesssim (t-s)^{\theta} |z_s|_{\cB{\gamma+\sigma}}
    \lesssim (t-s)^{\theta} \|z\|_{\infty,\gamma+\sigma},
\end{align*}
which was estimated in \eqref{est:infty:z}. \\
Finally, the last term can be estimated by 
\begin{align*}
    |S(t-s)y'_s \xx_{s,t}|_{\cB_{\gamma+\sigma-\theta}}
    &\lesssim (t-s)^{2\alpha} \|\xx \|_{2\alpha}
    |S(t-s)|_{\cL({\cB_{\gamma+\sigma-2\alpha},\cB_{\gamma+\sigma-\theta}})}
    |y'_s|_{\cB_{\gamma+\sigma-2\alpha}} \\
    &\lesssim (t-s)^{\theta} |y'_s|_{\cB_{\gamma+\sigma-2\alpha}} \\
    &\lesssim (t-s)^{\theta} \left\|y' \right\|_{\infty,\cB_{\gamma+\sigma-2\alpha}},
\end{align*}
which again was estimated in \eqref{est:hoelder:yprime}.\\
Summarizing we obtain the bounds 
\begin{align*}
    \left\|R^z \right\|_{\theta,\gamma+\sigma-\theta}
    \lesssim |y_0|_{\cB_{\gamma}} + |y'_0|_{\cB_{\gamma-\alpha}} + T^{\alpha-\sigma} \gubnorm{y}{y'}. 
\end{align*}
Putting all these estimates together in~\eqref{norm:z} proves the statement.
\end{proof}

The following result provides an estimate on the composition of a controlled rough path with a smooth function $G$ satisfying assumption (G). In order to avoid quadratic terms as in~\cite[Lemma 4.7]{GHN} we directly use the structure of the solution as specified in~\eqref{solution}.
\begin{lemma}\label{smooth:fuction}
Let $G$ satisfy assumption (G) and $(y,G(y))\in D^{2\alpha}_{X,\gamma}$. Then $(G(y),DG(y) G(y))\in D^{2\alpha}_{X,\gamma-\sigma}$ and the following bound is valid
\begin{align*}
    \gubnormpar{G(y)}{DG(y)G(y)}{\gamma-\sigma}
    \lesssim 1+\gubnorm{y}{y'}.
\end{align*}
\end{lemma}
\begin{proof}
By \eqref{g:norm}
\begin{align*}
    \gubnormpar{G(y)}{DG(y)G(y)}{\gamma-\sigma}
    &= \left\|G(y) \right\|_{\infty,\gamma-\sigma}
    + \left\|DG(y)G(y) \right\|_{\infty,\gamma-\alpha-\sigma}
    + \left\|DG(y)G(y) \right\|_{\alpha,\gamma-2\alpha-\sigma}\\
    &+ \big\|R^{G(y)} \big\|_{\alpha,\gamma-\alpha-\sigma}
    + \big\|R^{G(y)} \big\|_{2\alpha,\gamma-2\alpha-\sigma}.
\end{align*}
The first term can be bounded due to the boundedness of $DG$ by
\begin{align*}
    \left\|G(y) \right\|_{\gamma,\cB_{\gamma-\sigma}} 
    \lesssim 1 + \left\|y \right\|_{\infty,\gamma}
    \leq 1 + \gubnorm{y}{G(y)}.
\end{align*}
For the second term we have
\begin{align*}
    \left\|DG(y)G(y) \right\|_{\infty,\gamma-\alpha-\sigma}
    \leq \left\|DG(y) \right\|_{\infty,\cL(\cB_{\gamma-\alpha},\cB_{\gamma-\alpha-\sigma})} \left\|G(y) \right\|_{\infty,\gamma-\alpha}
    \lesssim \gubnorm{y}{G(y)}.
\end{align*}
The third term can be bounded due to assumption (G) by
\begin{align*}
   \left\|DG(y)G(y) \right\|_{\alpha,\gamma-2\alpha-\sigma} 
   \lesssim \left\|y \right\|_{\alpha,\gamma-\alpha}. 
\end{align*}
Now \eqref{est:hoelder:y} yields
\begin{align*}
     \left\|DG(y)G(y) \right\|_{\alpha,\gamma-2\alpha-\sigma}
     \lesssim \left\|G(y)\right\|_{\infty,\gamma-\alpha} + \left\|R^y \right\|_{\alpha,\gamma-\alpha}
     \lesssim \gubnorm{y}{G(y)}.
\end{align*}
For the remainder terms we use~\eqref{remainder} and write
\begin{align*}
    R^{G(y)}_{s,t} = G(y_t) - G(y_s) - DG(y_s) G(y_s) X_{s,t}
    &= \int_0^1 DG(y_s + r(y_t-y_s))~\txtd r~(y_t-y_s) - DG(y_s)  G(y_s) X_{s,t} \\
    &= \int_0^1 \big(DG(y_s + r(y_t-y_s))-DG(y_s)\big)~\txtd r~G(y_s) X_{s,t}\\ 
    &+ \int_0^1 DG(y_s + r(y_t-y_s))~\txtd r~R^y_{s,t}. 
\end{align*}
Again using the boundedness of $DG$, we obtain for the first remainder term
\begin{align*}
    \big\|R^{G(y)} \big\|_{\alpha,\gamma-\sigma-\alpha}
    &\lesssim \left\|DG(y) \right\|_{\infty,\cL(\cB_{\gamma-\alpha},\cB_{\gamma-\alpha-\sigma})} \left\|G(y) \right\|_{\infty,\gamma-\alpha}
    + \left\|DG(y) \right\|_{\infty,\cL(\cB_{\gamma-\alpha},\cB_{\gamma-\alpha-\sigma})} \big\|R^{y} \big\|_{\alpha,\gamma-\alpha} \\
    &\lesssim \gubnorm{y}{G(y)}.
\end{align*}
For the second remainder term we apply \eqref{est:g} and use the boundedness of $DG$.
\begin{align*}
 \big\|R^{G(y)} \big\|_{2\alpha,\gamma-2\alpha-\sigma}
 \lesssim \left\| y\right\|_{\alpha,\cB_{\gamma-\alpha}}     + \left\|DG(y) \right\|_{\infty,\cL(\cB_{\gamma-2\alpha},\cB_{\gamma-2\alpha-\sigma})} \left\| R^y\right\|_{2\alpha,\gamma-2\alpha}
\end{align*}
Again, using \eqref{est:hoelder:y} we obtain the bound 
\begin{align*}
    \big\|R^{G(y)} \big\|_{2\alpha,\gamma-2\alpha-\sigma}
    \lesssim \gubnorm{y}{G(y)}.
\end{align*}
\end{proof}
Putting the results of Lemma~\ref{lemma:intitial},  \ref{drift},~\ref{rough:integral} and~\ref{smooth:fuction} together we obtain.
\begin{corollary}\label{cor:gubnorm:bound}
Let $F$ and $G$ satisfy the assumptions (F) and (G) and let $(y,G(y))\in D^{2\alpha}_{X,\gamma}$ be the  solution of \eqref{equation:main} on the time interval $[0,T]$ with initial data $y_0\in\cB_\gamma$. Then the following estimate holds true
\begin{align}\label{bound}
    \gubnorm{y}{G(y)} \lesssim 1+\left|y_0 \right|_{\cB_\gamma}+T^{\eta} \gubnorm{y}{G(y)},
\end{align}
where $\eta:=(\alpha -\sigma)\wedge (1-\delta)$.
\end{corollary}
\begin{proof}
Since $(y,G(y))$ solves \eqref{equation:main} we have 
\begin{align*}
    \gubnorm{y}{G(y)}
    &\leq \gubnorm{S(\cdot)y_0}{0}
    + \gubnorm{\int_0^\cdot S(\cdot-r) F(y_r)~\txtd r}{0}\\
    &+ \gubnorm{\int_0^\cdot S(\cdot-r) G(y_r)~\dx_r}{G(y)}.
\end{align*}
Thus, 
\begin{align*}
    \gubnorm{y}{G(y)}
  &  \lesssim \left|y_0 \right|_{\cB_{\gamma}}
    +T^{1-\delta} (1+\left\|y \right\|_{\infty,\gamma}) 
    + \left|G(y_0) \right|_{\cB_{\gamma-\sigma}} 
    + \left|DG(y_0)G(y_0) \right|_{\gamma-\alpha-\sigma}\\
   & + T^{\alpha-\sigma}\gubnormpar{G(y)}{DG(y)G(y)}{\gamma-\sigma}.
\end{align*}
Finally Lemma \ref{smooth:fuction} as well as the boundedness of $DG$ yield
\begin{align*}
    \gubnorm{y}{G(y)}
    &\lesssim 1+\left|y_0 \right|_{\cB_{\gamma}}
    +T^{1-\delta} \gubnorm{y}{G(y)} 
    + T^{\alpha-\sigma}\gubnorm{y}{G(y)}\\
    &\lesssim 1+\left|y_0 \right|_{\cB_{\gamma}} + T^{(1-\delta)\wedge (\alpha-\sigma)}\gubnorm{y}{G(y)}.
\end{align*}
\end{proof}

Having the bound~\eqref{bound} for the $D^{2\alpha}_{X,\gamma}$-norm, we formulate an a-priori bound for the solution of~\eqref{equation:main}. \\
The proof of the following result relies on a concatenation argument~\cite[Lemma 5.6]{HN20}. Consequently, it is necessary to consider several norms on subintervals of $[0,T]$. Note that in contrast to the rest of this manuscript where the underlying time-interval is suppressed for a better readability, in the following lemma we indicate this additional time-dependence. The proof of the next statement relies on similar arguments to~\cite[Lemma 5.5]{HN20}.
\begin{lemma}\label{lemma:apriori}
Let $F$ and $G$ satisfy the assumptions (F) and (G) and let $(y,G(y))\in D^{2\alpha}_{X,\gamma}$ be the  solution of \eqref{equation:main} on the time interval $[0,T]$ with $T>0$ and initial data $y_0\in\cB_\gamma$. Let $r = 1 \vee \left|y_0 \right|_{\cB_{\gamma}}$. Then there exist constants $M_1,M_2>0$ such that
\begin{align*}
    \left\|y \right\|_{\infty,\gamma,[0,T]} \leq M_1 r e^{M_2 T}.
\end{align*}
\end{lemma}
\begin{proof}
For all $\Bar{T}\in (0,T]$ the restriction of $(y,G(y))$ on $[0,\Bar{T}]$ is a solution of \eqref{equation:main} on $[0,\Bar{T}]$. Thus, by Corollary \ref{cor:gubnorm:bound} we know that there exists a constant $C \geq 1$ such that  
\begin{align*}
      \gubnormpar{y}{G(y)}{\gamma,[0,\Bar{T}]} 
      \leq C \big(\rho+\Bar{T}^{\eta} \gubnormpar{y}{G(y)}{\gamma,[0,\Bar{T}]}\big),
\end{align*}
where $\eta=(\alpha -\sigma)\wedge (1-\delta)$.\\
 Hence, for all $\Bar{T}$ small enough such that $C\Bar{T}^\eta\leq \frac{1}{2}$ we obtain the bound 
\begin{align*}
  \left\| y \right\|_{\infty,\gamma,[0,\Bar{T}]}  \leq \gubnormpar{y}{G(y)}{\gamma,[0,\Bar{T}]} 
  \leq 2C r.
\end{align*}
If $CT^{\eta}\leq \frac{1}{2}$, the proposed statement holds true choosing $M_1\geq 2C$ and an arbitrary $M_2>0$.
Otherwise, we choose $N \in \IN$ (not necessarily unique) such that $\frac{1}{4}<C\big(\frac{T}{N}\big)^\eta \leq \frac{1}{2}$. This is possible because $\eta <1$. Then
\begin{align*}
     \left\|y \right\|_{\infty,\gamma,[0,\frac{T}{N}]}
     \leq 2C r.
\end{align*}
Further, a concatenation argument yields for all $k \in \left\{0, \ldots, N-1 \right\}$
\begin{align*}
    \left\|y \right\|_{\infty,\gamma,[\frac{k}{N}T,\frac{k+1}{N}T]} \leq (2C)^{k+1} r.
\end{align*}
Consequently,
\begin{align*}
    \left\|y \right\|_{\infty,\gamma,[0,T]}
    = \max_{k \in \{0,\ldots,N-1 \}} \left\|y \right\|_{\infty,\gamma,[\frac{k}{N}T,\frac{k+1}{N}T]} 
    \leq (2C)^N r.
\end{align*}
Finally, using that $\frac{1}{4}<C\big(\frac{T}{N}\big)^\eta $, consequently  $N < (4C)^{\frac{1}{\eta}} T$ we obtain
the claim with $M_1\geq (2C)^{(4C)^{\frac{1}{\eta}}}$ and $M_2 \geq \log(2C)$.
\end{proof}
Lemma \ref{lemma:apriori} guarantees that the solution of \eqref{equation:main} does not exhibit finite-time blow up under the assumptions (F) and (G). Therefore we formulate our main result.
\begin{theorem}\label{global:sol}
Let $T>0$, $F$ and $G$ satisfy the assumptions (F) and (G), $\bX=(X,\xx)$ be an $\alpha$-H\"older rough path and let $y_0\in\cB_\gamma$.
Then there exists a unique global-in-time solution $(y,G(y))\in D^{2\alpha}_{X,\gamma}([0,T])$ of \eqref{equation:main}. 
\end{theorem}
\begin{proof}
Let $r = 1 \vee |y_0|_{\cB_{\gamma}}$. Then, by Lemma \ref{lemma:apriori}, each solution of \eqref{equation:main} can be bounded by 
\begin{align*}
    \left\|y \right\|_{\infty,\gamma}
    \leq M_1 r e^{M_2 T} =: \Tilde{r}.
\end{align*}
Applying Theorem \ref{thm:local} with $\left|y_0 \right|_{\cB_{\gamma}}\leq \Tilde{r}$ yields the existence of $N=N(\alpha,\gamma,\Tilde{r},\rho_\alpha(\bX),F,G)$ such that there exists a unique local solution of \eqref{equation:main} on the time interval $[0,\frac{T}{N}]$ with initial condition $y_0$. \\
Since $\big|y_{\frac{T}{N}} \big|_{\cB_{\gamma}}\leq \Tilde{r}$, we furthermore obtain the existence of a unique local solution of \eqref{equation:main} on the time interval $[0,\frac{T}{N}]$ with initial data $y_{\frac{T}{N}}$. Concatenating both solutions provides a solution of \eqref{equation:main} on the time interval $[0,2\frac{T}{N}]$ with initial data $y_0$. 
Iterating this argument one can construct a solution on the whole time interval $[0,T]$.
\end{proof}
Based on our global well-posedness result, under suitable assumptions on the driving
rough path, we are able to construct a random dynamical system corresponding to~\eqref{equation:main}. To this aim we introduce some concepts from the theory of random dynamical systems~\cite{Arnold}. The following definition describes a model of the driving noise.

\begin{definition}\label{mds} 
Let $(\Omega,\mathcal{F},\mathbb{P})$ stand for a probability space and 
$\theta:\mathbb{R}\times\Omega\rightarrow\Omega$ be a family of 
$\mathbb{P}$-preserving transformations (i.e.,~$\theta_{t}\mathbb{P}=
\mathbb{P}$ for $t\in\mathbb{R}$) having the following properties:
\begin{description}
		\item[(i)] The mapping $(t,\omega)\mapsto\theta_{t}\omega$ is 
		$(\mathcal{B}(\mathbb{R})\otimes\mathcal{F},\mathcal{F})$-measurable, where 
		$\mathcal{B}(\cdot)$ denotes the Borel sigma-algebra;
		\item[(ii)] $\theta_{0}=\textnormal{Id}_{\Omega}$;
		\item[(iii)] $\theta_{t+s}=\theta_{t}\circ\theta_{s}$ for all 
		$t,s,\in\mathbb{R}$.
\end{description}
Then the quadrupel $(\Omega,\mathcal{F},\mathbb{P},(\theta_{t})_{t\in\mathbb{R}})$ 
is called a metric dynamical system.
\end{definition}

\begin{definition}
\label{rds} 
A continuous random dynamical system on a separable Banach space $\cX$ over a metric dynamical 
system $(\Omega,\mathcal{F},\mathbb{P},(\theta_{t})_{t\in\mathbb{R}})$ 
is a mapping $$\varphi:[0,\infty)\times\Omega\times \cX\to \cX,
\mbox{  } (t,\omega,x)\mapsto \varphi(t,\omega,x), $$
which is $(\mathcal{B}([0,\infty))\otimes\mathcal{F}\otimes
\mathcal{B}(\cX),\mathcal{B}(\cX))$-measurable and satisfies:
	\begin{description}
		\item[(i)] $\varphi(0,\omega,\cdot{})=\textnormal{Id}_{\cX}$ 
		for all $\omega\in\Omega$;
		\item[(ii)]$ \varphi(t+\tau,\omega,x)=
		\varphi(t,\theta_{\tau}\omega,\varphi(\tau,\omega,x)), 
		\mbox{ for all } x\in \cX, ~t,\tau\in[0,\infty),~\omega\in\Omega;$
		\item[(iii)] $\varphi(t,\omega,\cdot{}):\cX\to \cX$ is 
		continuous for all $t\in[0,\infty)$ and all $\omega\in\Omega$.
	\end{description}
\end{definition}

The second property in Definition~\ref{rds} is referred to as the 
cocycle property. The generation of a random dynamical system from an It\^{o}-type stochastic partial differential equation (SPDE) has been a long-standing open problem, since Kolmogorov's theorem breaks down for random fields parametrized by infinite-dimensional Banach spaces.~As a consequence it is not obvious how to obtain a random dynamical system from an SPDE, since its solution is defined almost surely, which contradicts the cocycle property. Particularly, this means that there are exceptional sets which depend on the initial condition and it is not clear how to define a random dynamical system if more than countably many exceptional sets occur. This issue does not occur in a pathwise approach. Under suitable assumptions on the coefficients, rough path driven equations generate random
dynamical systems provided that the driving rough path forms a rough path cocycle, as established in~\cite{BRiedelScheutzow}.\\
Let $(\Omega,\mathcal{F},\mathbb{P},(\theta_{t})_{t\in\mathbb{R}})$ be a metric dynamical system as in Definition~\ref{mds}. We say that 
\begin{align*}
    \mathbf{X}=(X,\xx):\Omega\to C^{\alpha}_{\text{loc}}([0,\infty);\mathbb{R}^d) \times C^{2\alpha}_{\text{loc}}([0,\infty);\mathbb{R}^{d\times d}) 
\end{align*}
is a continuous ($\alpha$-H\"older) rough path cocycle if $\mathbf{X}|_{[0,T]}$ is a continuous $\alpha$-H\"older rough path for every $T>0$ and for every $\omega\in\Omega$ and the following cocycle property holds true for every $s,t\in[0,\infty)$ and $\omega\in\Omega$
\begin{align*}
    &X_{s,s+t}(\omega)= X_t(\theta_s\omega)\\
    &\xx_{s,s+t}(\omega)=\xx_{0,t}(\theta_s\omega).
\end{align*}
According to~\cite[Section 2]{BRiedelScheutzow} rough path lifts of various stochastic processes define cocycles.
These include Gaussian processes with stationary increments under certain
assumption on the covariance function~\cite[Chapter 10]{FrizHairer} and particulary apply to the fractional Brownian motion with Hurst index $H>\frac{1}{4}$. Recall that here we fixed the $\alpha$-H\"older regularity of the rough path $\alpha\in(\frac{1}{3},\frac{1}{2})$, consequently the results obtained apply to fractional Brownian motion for $H\in(\frac{1}{3},\frac{1}{2}]$. \\

Due to the previous deliberations and Theorem~\ref{global:sol} we immediately infer.
\begin{theorem}\label{thm:rds}
Under the assumptions of Theorem~\ref{global:sol}, the solution operator of~\eqref{equation:main} generates a continuous random dynamical system on $\cB_\gamma$. 
\end{theorem}
\section{Examples}\label{sec:ex}
\begin{example}
Let $k\in\mathbb{N}$, $p>1$ and $\mathbb{T}^d$ stands for the $d$-dimensional torus. Furthermore, let $\mathbf{X}=(X,\xx)$ be the rough path lift of a fractional Brownian motion $X$ with Hurst index $H\in(1/3,1/2]$ and let $\sigma<H$. We consider the semilinear parabolic rough partial differential equation
\begin{align}\label{equation:ex}
    \begin{cases}
  \txtd y_t (x) = [\Delta y_t(x) + f(y_t(x))]~\txtd t + g(x) (-\Delta)^\sigma y_t(x)~\dx_t\\
  y_0(x) \in H^{k,p}(\mathbb{T}^d).
    \end{cases}
\end{align}
In this case it is well-known that $\cB_\gamma=H^{k+2\gamma,p}(\mathbb{T}^d)$ for $\gamma\in\mathbb{R}$, where $H^{k,p}(\mathbb{T}^d)$ are Bessel potential spaces~\cite[Chapter 16]{Yagi}.
Moreover $(-\Delta)^\sigma:\cB_\gamma\to \cB_{\gamma-\sigma}$ for all $\gamma\in \mathbb{R}$ and the multiplication with a smooth function $g$ is a smooth operation from $\cB_{\gamma-\sigma}$ into itself. Therefore $G(y):=g(x)(-\Delta)^\sigma y$ satisfies assumption (G). Choosing a nonlinear term $f$
satisfying assumption (F), we obtain that~\eqref{equation:ex} has a global-in-time solution. 
\end{example}
More general, an example for linear operator $A$ in~\eqref{equation:main} is  the $L^p$-realization (for $1<p<\infty$) of a strongly elliptic operator with suitable boundary conditions~\cite[Section 7.3]{Pazy}.
\begin{example}
Let $\cO$ stand for an open bounded domain with smooth boundary in $\mathbb{R}^d$ and define
\begin{align}
    A(x,D) u = \sum\limits_{k,l=1}^n \frac{\partial }{\partial x_k} \big(a_{k,l}(x) \frac{\partial u}{\partial x_l}\big), 
\end{align}
where the coefficients $a_{k,l}(x)=a_{l,k}(x)$ are real-valued and continuously differentiable on $\overline{\cO}$ and $A(x,D)$ is strongly elliptic, i.e. there exists a constant $C$ such that
\begin{align*}
    \sum\limits_{k,l=1}^n a_{k,l}(x)\xi_k\xi_l \geq C \sum\limits_{k=1}^n \xi^2_k 
\end{align*}
for all $\xi_k\in\mathbb{R}$, $1\leq k \leq n$. In this case, one introduces the $L^p$-realization of $A(x,D)$ with Dirichlet boundary conditions as follows
\begin{align*}
   & A y = A(x,D) y,~~\text{ for } y \in D(A)\\
   & D(A)=H^{2,p}(\cO)\cap H^{1,p}_0(\cO).
\end{align*}
Then the operator $A$ generates an analytic $C_0$-semigroup in $L^p(\cO)$ and $\cB_\gamma=H^{2\gamma}(\cO)$~\cite[Theorem 3.6]{Pazy} and~\cite[Chapter 16]{Yagi}.
\end{example}
\begin{remark}
The theory developed in this work can be extended to time-dependent operators $A(t)$ generating parabolic evolution families $(U(t,s))_{t\geq s}$ on a separable Banach space $\cB$. Analogously, one can work with the monotone family of interpolation spaces $(\cB_\gamma)_{\gamma\in\mathbb{R}}$ as in Definition~\ref{raum},  satisfying for $t>s$ similar estimates to~\eqref{hg:1} and~\eqref{hg:2}, i.e. 
\begin{align}
|(U(t,s)-\text{\em Id}) x|_{\gamma}&\lesssim |t-s|^\sigma |x|_{\gamma+\sigma}\\
 |U(t,s)x|_{\gamma+\sigma}&\lesssim |t-s|^{-\sigma}|x|_\gamma.
\end{align}
For further applications see~\cite{GHN,Veraar}. 
\end{remark}


\end{document}